\documentclass[reqno]{amsart}

\theoremstyle{plain}
\newtheorem{theorem}{Theorem}[section]
\newtheorem{lemma}[theorem]{Lemma}
\newtheorem{corollary}[theorem]{Corollary}
\newtheorem{proposition}[theorem]{Proposition}

\newtheorem{ttheorem}{Theorem}

\theoremstyle{definition}
\newtheorem{definition}[theorem]{Definition}

\newtheorem{example}[theorem]{Example}

\theoremstyle{remark}
\newtheorem*{remark}{Remark}

\newtheorem*{notation}{Notation}
\newtheorem*{question}{Question}

\usepackage{amsmath}
\usepackage{amsfonts}
\usepackage{amssymb}
\usepackage{amscd}
\usepackage{stmaryrd}
\usepackage[matrix,arrow,cmtip,rotate,curve,arc]{xy}
\SelectTips{cm}{}
\input xy
\xyoption{all}
\usepackage{graphics}
\usepackage{graphicx}
\usepackage{fancyhdr}
\usepackage{ushort}

\newcommand{\bbP}{\mathbb P}

\newcommand{\bbR}{\mathbb R}

\newcommand{\bbZ}{\mathbb Z}

\newcommand{\mcE}{\mathcal E}

\newcommand{\mcI}{\mathcal I}

\newcommand{\mcM}{\mathcal M}
\newcommand{\mcO}{\mathcal O}
\newcommand{\mcP}{\mathcal P}

\newcommand{\mcR}{\mathcal R}

\newcommand{\mcX}{\mathcal X}

\newcommand{\mfX}{\mathfrak X}

\newcommand{\mfm}{\mathfrak m}
\newcommand{\mfn}{\mathfrak n}

\renewcommand{\hom}{\mathop\mathsf{hom}\nolimits}

\newcommand{\xra}[1]{\xrightarrow{#1}}

\newcommand{\ra}{\rightarrow}

\newcommand{\xlra}[1]{\xrightarrow{\ #1\ }}

\newcommand{\lra}{\longrightarrow}

\newdir{c}{{}*!/-5pt/@^{(}}
\newdir{d}{{}*!/-5pt/@_{(}}
\newdir{ >}{{}*!/-5pt/@{>}}
\newdir{s}{{}*!/+10pt/@{}}

\newcommand{\codim}{\mathop\mathrm{codim}\nolimits}

\newcommand{\Diff}{\mathrm{Diff}}

\newcommand{\germ}{\mathrm{germ}}

\newcommand{\Hom}{\mathop\mathrm{Hom}\nolimits}

\newcommand{\id}{\mathrm{id}}

\newcommand{\lspan}{\mathop\mathrm{span}}

\newcommand{\SP}{\mathop\mathrm{SP}\nolimits}
\newcommand{\spec}{\mathop\mathrm{sp}\nolimits}

\newcommand{\SurHom}{\mathop\mathrm{SurHom}\nolimits}

\newcommand{\wspec}{\mathrm{wsp}}

\newlength{\hlp}

\newcommand{\rightbox}[2]{\settowidth{\hlp}{$#2$}\makebox[\hlp][r]{${#1}{#2}$}}

\entrymodifiers={+!!<0pt,\the\fontdimen22\textfont2>}

\usepackage{ushort}

\title{The space of ideals of $C^\infty(M,\bbR)$}
\author{Luk\'a\v{s} Vok\v{r}\'inek}
\address{Department of Mathematics and Statistics\\Masaryk University\\Kotl\'a\v rsk\'a 2\\611 37 Brno\\Czech Republic}
\email{koren@math.muni.cz}
\subjclass[2000]{Primary 53; Secondary 57}

\begin{document}

\begin{abstract}
The paper is concerned with defining a topology on the set of ideals of codimension $d$ of the algebra $C^\infty(M,\bbR)$ with $M$ being a compact smooth manifold. Its main property is that it is compact Hausdorff and it contains as a subspace the configuration space of $d$ distinct unordered points in $M$ and therefore provides a ``compactification'' of this configuration space. It naturally forms a space over the symmetric product $\SP_d(M)$ and as such has a very rich structure. Moreover it is covered by a (naturally defined) set of charts in which it is semialgebraic.
\end{abstract}

\maketitle

\setcounter{section}{-1}

\section{Introduction}

In \cite{Glaeser}, Glaeser showed that for any parallelepiped $P\subseteq\bbR^m$ the set of ideals in $C^\infty(P,\bbR)$ of a fixed finite codimension can be given a topology using polynomials on $\bbR^m$ and proved that this space is compact Hausdorff. The deficiency of his approach is that this topology is then only seen to be invariant with respect to affine maps. We manage to avoid polynomials and topologize the set of ideals of the algebra of smooth functions on a smooth manifold (which would not be possible with \cite{Glaeser} for the lack of diffeomorphism invariancy).

In this paper let $M$ be a smooth manifold. In the introduction (and the abstract too) we assume for simplicity $M$ to be compact. Let $\mcR=C^\infty(M,\bbR)$ the algebra of smooth functions on $M$. We
will be considering ideals of $\mcR$ of a fixed finite codimension $d$ over $\bbR$.
Our first goal is to endow the set $\mcM_d$ of all such
ideals with a topology and study its properties. It will be compact Hausdorff. Together
with $\mcM_d$ we will also topologize
\[\mcE_d=\{(I,f+I)\ |\ I\in\mcM_d,\ f+I\in\mcR/I\}\]
in such a way that the canonical projection map $\mcE_d\ra\mcM_d$
becomes a vector bundle. The map $\mcM_d\times\mcR\lra\mcE_d$
sending $(I,f)$ to $(I,f+I)$ is a (continuous) homomorphism of
vector bundles. This is certainly a property one would require
from any such topology.

In the introduction to his thesis \cite{Goodwillie}, Goodwillie writes that $\mcM_d$ should be ``semialgebraic relative to $M$''. In our construction of $\mcM_d$ this has a clear meaning and our second goal will be to prove Goodwillie's claim.

For any set $Y\subseteq M$ consisting of $d$
points one has an ideal
\[\mfm_Y:=\{f\in\mcR\ |\ f(y)=0\textrm{ for all }y\in Y\}\]
and easily $\mfm_Y\in\mcM_d$. In this way one has as a subset in
$\mcM_d$ the configuration space $M^{[d]}$ of $d$ distinct
unordered points in $M$,
\[M^{[d]}\hookrightarrow\mcM_d\]
For our topology on $\mcM_d$ this inclusion map is an
embedding and therefore one can think of $\mcM_d$ as a
``compactification'' of $M^{[d]}$.

There is an inverse procedure of associating to each $I\in\mcM_d$
an unordered $d$-tuple of points in $M$. However for ideals not of
the form $\mfm_Y$ these points are not distinct and
therefore this procedure yields a map
\[\wspec:\mcM_d\ra M^d/\Sigma_d=\SP_d\]
We call $\wspec(I)$ the weighted spectrum of $I$.

Let us try to indicate now how we construct the topology on
$\mcM_d$. If $\mcR$ was finite dimensional we could simply think
of $\mcM_d$ as a subset of the Grassmannian manifold $G_d(\mcR)$
of linear subspaces of $\mcR$ of codimension $d$. This is of
course never the case. On the other hand suppose that there
is a finite dimensional linear subspace $F\subseteq\mcR$ that is
transverse to each $I\in\mcM_d$. Then taking the intersection of each $I\in\mcM_d$ with $F$
produces a map $\mcM_d\ra G_d(F)$ from $\mcM_d$ to the Grassmannian manifold of codimension $d$ linear subspaces of $F$. If this map is moreover injective one could give $\mcM_d$ the subspace topology. With the goal of giving $\mcM_d$ a structure of a fibrewise semialgebraic set we rather think of $\mcM_d$ as embedded into $\SP_d\times G_d(F)$. Nevertheless the first key step is still to find a transversal $F$ for which the map $\mcM_d\ra G_d(F)$ is injective.
 
This is done in the first section and the results here are both well-known and quite elementary. It starts by describing the structure of ideals of $\mcR$ so that one is able to understand the transversality condition. At the end of the section we produce a transversal $F$.

In the second section we compare the topologies induced via two such transversals $F\subseteq F'$. The main result states that the (partially defined) map
\[\SP_d\times G_d(F')\lra\SP_d\times G_d(F)\]
sending $(Y,L)$ to $(Y,F\cap L)$ induces a homeomorphism between the images of $\mcM_d$. Therefore the above topology is independent of the choice of the transversal and is our topology on $\mcM_d$. This is proved in Theorem~\ref{int_independence_of_topology}. We then proceed to describe in Theorem~\ref{int_fibrewise_semialgebraicity} the structure on $\mcM_d$ of a semialgebraic set. There are special charts on $\SP_d\times G_d(F)$ coming from pairs of a chart on $M$ and a ``compatible'' transversal $F$. In these charts $\mcM_d$ is semialgebraic. This situation is in the spirit very much like that of ``semialgebraic'' subsets of the jet bundle $J^r(M,\bbR^k)$ coming from  $\Diff(\bbR^m)$-invariant semialgebraic subsets of $J^r(\bbR^m,\bbR^k)$ which are also semialgebraic in all charts coming from $M$. The formal definition is given in Definition~\ref{definition_semialgebraic}.

The third section is dedicated to studying the topology of
$\mcM_d$. Two theorems are proved, the first of which states that
the inclusion of the configuration space into $\mcM_d$ is a smooth
embedding as was already mentioned above. The second theorem (in fact a generalization of the first) is
concerned with the set of all ideals $I$ with a fixed isomorphism
type of the quotient algebra $\mcR/I$. These ideals possess a
structure of a smooth manifold (as was proved in \cite{Alonso})
and the content of our second theorem is that they form an
immersed submanifold in $\mcM_d$. As a set $\mcM_d$ is then a
disjoint union of these submanifolds. In general there is an infinite number of them.

\section{The structure of ideals of $\mcR$}

\begin{notation}
We recall the notation from the introduction. Let $M$ be a smooth manifold (possibly non-compact and with boundary or even with corners) and let $\mcR=C^\infty(M,\bbR)$ denote the algebra of all smooth functions on $M$.
\end{notation}

\begin{remark}
Most of the content of the paper could be generalized to the case of submodules of a fixed finitely generated $\mcR$-module of a fixed finite codimension. This is treated quite in detail in the author's thesis \cite{Thesis}.
\end{remark}

The aim of this section is to prove the following theorem. The main reference is the authors's thesis\footnote{Only the case of a compact manifold $M$ is treated in detail. The extension to non-compact $M$ is provided by the fact that maximal ideals not corresponding to points of $M$ have always infinite codimension (the proof of which can be also found there in section 1.5.).} \cite{Thesis} although all the results are more or less well-known.

\begin{theorem} \label{int_existence_of_transversal}
There is a finite dimensional linear subspace $F\subseteq\mcR$ which is transverse to all ideals $I\subseteq\mcR$ of a fixed finite codimension $d$.
\end{theorem}

The idea of the proof is pretty simple. Every ideal of a finite codimension contains an ideal of the form $(\mfm_{y_1})^{k_1}\cdots(\mfm_{y_n})^{k_n}$ where $\mfm_y$ are the maximal ideals. For ideals of this form the transversality ammounts to a certain interpolation property which is easily satisfied for example by polynomials. We proceed with a more detailed proof.

\begin{definition}
For an ideal $I\subseteq\mcR$ the \textit{spectrum} of $I$ is
defined to be the closed subset
\[\spec(I):=\bigcap_{f\in I}f^{-1}(0)=\{x\in M\ |\ \forall f\in I:f(x)=0\}\]
\end{definition}

\begin{example}
For a point $x\in M$ we define the following two ideals 
\[\mfm_x:=\{f:M\rightarrow\bbR\ |\ f(x)=0\} \qquad \mfn_x:=\{f:M\rightarrow\bbR\ |\ \germ_xf=0\}\]
Clearly $\spec(\mfm_x)=\spec(\mfn_x)=\{x\}$.
\end{example}

The main structural result for ideals of finite codimension is the following.

\begin{theorem}
Every ideal $I\subseteq\mcR$ of finite codimension has a unique (up to permutation of the factors) decomposition as an intersection
\[I=I_1\cap\cdots\cap I_n\]
of ideals $I_i\subseteq\mcR$ with the following two properties:
\begin{itemize}
\item the spectrum of each $I_i$ consists of a single point $y_i$
\item the points $y_i$ are all different
\end{itemize}
In this decomposition $I_i=I+\mfn_{y_i}$, ideals $I_i$ are in general position, i.e.
\[I_i+(I_1\cap\cdots\cap\widehat{I_i}\cap\cdots\cap I_n)=\mcR\]
and $y_i$ are exactly the points of the spectrum $\spec(I)$. This decomposition is called the \emph{primary decomposition} of $I$.
\end{theorem}

\begin{proof}
The statement is rather classical and a complete proof can be found in \cite{Thesis}. Namely in Lemma~1.8.~the uniqueness of the decomposition is proved together with its form $I_i=I+\mfn_i$. In the next two paragraphs following this lemma the finiteness of the spectrum is shown together with its implication towards the existence of the decomposition. General position of $I_i$'s is proved in Corollary~1.7.
\end{proof}

Denoting the spectrum of $I$ by $\{y_1,\ldots,y_n\}$ as in the statement the general position of the decomposition yields the following formula for codimensions.
\[\dim\mcR/I=\dim\mcR/(I+\mfn_{y_1})+\cdots+\dim\mcR/(I+\mfn_{y_n})\]
We set $k_i=\dim\mcR/(I+\mfn_{y_i})$ and observe that the
spectrum of $I$ has more structure if $I$ has finite codimension:
each point $y_i$ in the spectrum has associated a weight $k_i$
with it.

We fix an integer $d$ and denote the collection of all ideals of codimension~$d$ by $\mcM_d$. We define a space $\SP_d$, the symmetric product, where these more structured spectra will be defined simply by
\[\SP_d:=M^d/\Sigma_d\]
and give it the quotient topology. The space $\SP_d$ consists of unordered
collections of $d$ not necessarily distinct points in $M$. If
$y_1,\ldots,y_n$ are all the points in such a collection $Y$ and if
each $y_i$ appears in it exactly $k_i$-times then we say that $k_i$
is the weight of $y_i$ and use an alternative notation
\[Y=\{y_1^{k_1},\ldots,y_n^{k_n}\}\]
We call $|Y|:=\{y_1,\ldots,y_n\}$ the \textit{support} of $Y$. There
is a weight function $|Y|\rightarrow\bbZ^+$ associating to each
point $y_i$ its weight $k_i$ and by the definition the total weight
$\sum k_i$ is $d$. Therefore we also call $Y$ a set of points with
weights.

Hence with every ideal $I\subseteq\mcR$ of finite codimension
$d$ there is associated a canonical set of points with weights
$Y\in\SP_d$ whose support is the spectrum of $I$ (and whose total
weight is $d$). It is called the \textit{weighted spectrum} of $I$,
\[\wspec:\mcM_d\ra\SP_d\]
On the other hand if $Y=\{y_1^{k_1},\ldots,y_n^{k_n}\}\in\SP_d$
we define an ideal $\mfm_Y\subseteq\mcR$ by the formula
 \[\mfm_Y=(\mfm_{y_1})^{k_1}\cap\cdots\cap(\mfm_{y_n})^{k_n}=\{f\in\mcR\ |\ j^{k_1-1}_{y_1}f=\cdots=j^{k_n-1}_{y_n}f=0\}\]
The following simple consequence of Nakayama's lemma is proved in \cite{Thesis} as Lemma~1.9.

\begin{lemma} \label{lemma_polynomial_submodule_containment}
Let $I\subseteq\mcR$ be an ideal of finite codimension with $\wspec(I)=Y$. Then $\mfm_Y\subseteq I$.\qed
\end{lemma}

Let $F$ be a linear subspace of $\mcR$. We write $F\pitchfork\mcM_d$ if $F$ is transverse to every element of $\mcM_d$, i.e. to every ideal of $\mcR$ of codimension $d$, and $F\pitchfork\mcP_d$ if $F$ is transverse to all ``polynomial'' ideals $\mfm_Y$, whenever $Y\in\SP_d$.

We will now construct a finite dimensional linear subspace $F\pitchfork\mcP_d$ (and hence also $F\pitchfork\mcM_d$ by Lemma~\ref{lemma_polynomial_submodule_containment}). First we find this subspace locally.

\begin{lemma} \label{int_local_existence_of_transversal_for_R}
For $M=\bbR^m$ there is a finite dimensional linear subspace
$F\subseteq\mcR$ satisfying $F\pitchfork\mcP_d$.
\end{lemma}

\begin{proof}
This is the classical interpolation theory for polynomials. A proof can be found for example in \cite{Thesis} as Lemma~1.13.
\end{proof}

\begin{proof}[Proof of Theorem~\ref{int_existence_of_transversal}]
Let us choose an embedding $\iota:M\hookrightarrow V$ of $M$ into a Euclidean space $V$ and note that the canonical map $\iota^*:C^\infty(V,\bbR)\lra C^\infty(M,\bbR)$ is surjective, i.e. any smooth map $M\ra\bbR$ can be extended to $V$. By Lemma~\ref{int_local_existence_of_transversal_for_R} there exists a finite dimensional linear subspace $F'\subseteq C^\infty(V,\bbR)$ such that $F'\pitchfork\mcP_d$ on $V$. Then we claim that $F=\iota^*F'=\{f\iota\ |\ f\in F'\}$ satisfies $F\pitchfork\mcP_d$ on $M$. This is verified by the following diagram where $J_Y(M,\bbR)=J_{y_1}^{k_1-1}(M,\bbR)\times\cdots\times J_{y_n}^{k_n-1}(M,\bbR)$.
\[\xymatrixnocompile{
F' \ar@{c->}[r] \ar[d]^{\iota^*} & C^\infty(V,\bbR) \ar@{->>}[r] \ar@{->>}[d]^{\iota^*} & J_{\iota(Y)}(V,\bbR) \ar@{->>}[d]^{\iota^*} \\
\rightbox{F={}}{\iota^*F'} \ar@{c->}[r] & C^\infty(M,\bbR) \ar@{->>}[r] & J_Y(M,\bbR)
}\]
The composition across the top row is surjective and therefore so must be the composition across the bottom one.
\end{proof}

\section{The topology and geometry of $\mcM_d$}

In this section we define a topology on $\mcM_d$ using a transversal $F$ of Theorem~\ref{int_existence_of_transversal} and prove that in fact it is independent of the choice of this transversal. There is the following order-preserving map
\[U_F:\{\textrm{ideals of $\mcR$}\}\lra\{\textrm{linear subspaces of $F$}\}\]
sending $I\subseteq\mcR$ to its intersection $F\cap I$ with $F$.

\begin{proposition} \label{int_injectivity_proposition}
If $F\pitchfork\mcM_{d+1}$ then for ideals $I,I'\subseteq\mcR$ of codimension at most $d$ the following holds
\[I\subseteq I'\quad\Longleftrightarrow\quad F\cap I\subseteq F\cap I'\]
In particular $U_F|_{\mcM_d}$ is injective. Moreover $F$ with these properties exists.
\end{proposition}

\begin{proof}
Assume for contradiction that $F\cap I\subseteq F\cap I'$ but also $I\cap I'\subsetneqq I$. Let $I''$ be a maximal ideal of $I$ containing $I\cap I'$ which exists by the finite codimension of $I\cap I'$. By maximality $I/I''\cong\mcR/\mfm_y$ for some maximal ideal $\mfm_y$ and in particular $I''\in\mcM_{d+1}$, hence $F\pitchfork I''$. By transversality then
\[(F\cap I)\cap(F\cap I')=F\cap(I\cap I')\subseteq F\cap I''\subsetneqq F\cap I,\]
a contradiction.
\end{proof}

This proposition allows us to use $U_F$ for putting a topology on $\mcM_d$. As we explained in the introduction it is more convenient to induce the topology on $\mcM_d$ from the inclusion
\[\bar U_F=(\wspec,U_F):\mcM_d\lra\SP_d\times G_d(F)\]
This has the advantage of splitting the two ``directions'' in $\mcM_d$, the geometry and the algebra. These correspond respectively to coordinates in $\SP_d$ and $G_d(F)$. Moreover in this way $\mcM_d$ is a space over $\SP_d$ straight from the definition.

To explain the statement of the main theorem recall that there is a canonical vector bundle $\mcE_d$ over $\mcM_d$ whose fibre over $I\in\mcM_d$ is the quotient algebra $\mcR/I$. Obviously it is (as a set) the pullback along $U_F$ of the canonical $d$-dimensional vector bundle $E_d(F)$ over $G_d(F)$ with fibre $F/L$ over $L\subseteq F$. This allows us to induce a topology on $\mcE_d$ from $\mcM_d$ and $E_d(F)$. It is also useful to think of $E_d(F)$ as a subset of $G_d(F)\times F$ and thus of $\mcE_d$ as a subset of $\SP_d\times G_d(F)\times F$ by identifying $F/L\cong L^\perp$ for some scalar product on $F$.

\begin{ttheorem} \label{int_independence_of_topology}
The topology on $\mcM_d$ does not depend on the choice of $F$ as long as $\bar U_F$ is injective. The weighted spectrum map $\wspec:\mcM_d\ra\SP_d$ is continuous and proper ($\mcM_d$ is a closed subset of $\SP_d\times G_d(F)$). The topology on $\mcE_d$ is also independent of $F$ and the canonical map $\mcM_d\times\mcR\ra\mcE_d$ given on the fibre over $I$ by the projection $\mcR\ra\mcR/I$ is a continuous quotient map of vector bundles over $\mcM_d$.
\end{ttheorem}

We will now explain in what sense $\mcM_d$ possesses a (singular) smooth structure. Firstly a function $\SP_d\ra\bbR$ is said to be smooth if the composition $M^d\ra\SP_d\ra\bbR$ is smooth (thus we are thinking of $\SP_d$ as an orbifold). As a subset of $\SP_d\times G_d(F)$ we may now consider on $\mcM_d$ the sheaf of restrictions of smooth functions. We will show in Theorem~\ref{int_fibrewise_semialgebraicity} that this structure is independent of $F$.

To endow $\mcM_d$ with a richer structure observe that one can embed $\SP_d\times G_d(F)$ in a (trivial) projective bundle $\SP_d\times\bbP(\Lambda^{r-d}F)$ where $r=\dim F$. The Pl\"ucker embedding $G_d(F)\ra\bbP(\Lambda^{r-d}F)$ sends
\[\lspan(f_1,\ldots,f_{r-d})\longmapsto[f_1\wedge\cdots\wedge f_{r-d}]\]
Therefore one can also speak about a smooth fibrewise algebraic structure given by the sheaf of functions which are regular rational in the $G_d(F)$ direction and whose coefficients are smooth in the $\SP_d$ direction, $U\times U'\mapsto C^\infty(U)\otimes\mcO(U')$.

We will now address the question of how bad the singularities of $\mcM_d$ are.

\begin{definition}\label{definition_semialgebraic}
For $M=\bbR^m$ there is a canonical choice of a transversal $F\subseteq\mcR$, namely the set of all polynomial functions of degree less than $d$. As a subspace of $(\bbR^m)^d/\Sigma_d\times G_d(F)$ we may ask that $\mcM_d$ is semialgebraic\footnote{More precisely the pullback to $(\bbR^m)^d\times G_d(F)$ should be semialgebraic. For other (non-polynomial) transversals the semialgebraicity should not be expected.}. Assuming this and covering any smooth manifold by charts this makes $\mcM_d$ locally semialgebraic (in the special charts described above). We express this by saying that $\mcM_d$ is \emph{chart-wise semialgebraic}.
\end{definition}

\begin{ttheorem} \label{int_fibrewise_semialgebraicity}
There is a well defined smooth (and fibrewise algebraic) structure on $\mcM_d$ as a subset of $\SP_d\times G_d(F)$ which is independent of $F$ as long as $\bar U_F$ is injective. It is a closed chart-wise semialgebraic subset and the same is true for $\mcE_d$.
\end{ttheorem}

We will now introduce our main tool in this section. By a \emph{continuous interpolation map} we understand a continuous map
\[A:\SP_d\times\mcR\ra F\]
(for some $F\pitchfork\mcP_d$), linear in the first variable and such that $A(Y,f)\equiv f$ modulo $\mfm_Y$. By definition for each $Y\in\SP_d$ and $f\in\mcR$ such an $A(Y,f)$ exists but there is no obvious choice and in particular it is not clear that there is a continuous way of choosing it. We say that $A$ is a \emph{smooth interpolation map} if moreover for each $f\in\mcR$ the map $\SP_d\xra{A(-,f)}F$ is smooth. There is an analogous notion of a \emph{polynomial interpolation map} for $M=\bbR^m$ in which $A(-,f)$ is required in addition to be polynomial when $f$ is polynomial.

When $F$ consists of polynomials on $M=\bbR^m$ a preferred choice of an interpolation was found by Kergin in \cite{Kergin} on the existential level and in \cite{MM} explicitly. We give a concrete description in the non-polynomial global case, main result being the following.

\begin{ttheorem} \label{int_continuous_interpolation}
Let $M$ be a smooth manifold and $F\pitchfork\mcP_d$. For each $Y\in\SP_d$ there exists a smooth interpolation map $A$ for which $A(Y,f)=0$ whenever $f\in\mfm_Y$.
\end{ttheorem}

The proof is postponed to the appendix.

\begin{remark}
We note that According to Example~\ref{example_Kergin_interpolation} if $F$ consists of polynomials of degree less than $d$ on $M=\bbR^m$ then the constructed interpolation map is the Kergin interpolation and is polynomial.
\end{remark}

Before we proceed further we state an easy corollary which will be crucial for the proof of the properness of $\wspec$.

\begin{corollary} \label{int_convergence_in_Sd}
Let $Y_p\in\SP_d$ ($p=1,2,\ldots$) be a sequence converging to $Y$. Then every $f\in F\cap \mfm_Y$ is a limit of some sequence $f_p\in F\cap \mfm_{Y_p}$.
\end{corollary}

\begin{proof}
Take $f_p=f-A(Y_p,f)$. Then $f_p\ra f$ since $A(Y,f)=0$.
\end{proof}

\begin{proof}[Proof of Theorem~\ref{int_independence_of_topology}, the independence part, and Theorem~\ref{int_fibrewise_semialgebraicity}]
For a pair of linear subspaces $F\subseteq F'$ as in the statement we denote by $G_d(F',F)$ the open subset of $G_d(F')$ consisting of those $L\subseteq F'$ which are transverse to $F$. The map $G_d(F',F)\lra G_d(F)$ sending $L$ to $F\cap L$ is a smooth affine bundle and hence so is
\[\SP_d\times G_d(F',F)\lra\SP_d\times G_d(F)\]
Moreover it respects the inclusions $\bar U_F,\bar U_{F'}$. Under the assumption $F\pitchfork\mcP_d$ we will now construct a smooth section of this bundle with $\bar U_{F'}\mcM_d$ in its image. It is given by sending $L\in G_d(F)$ in the fibre over $Y$ to
\[L+(\id-A(Y,-))(F^\perp)\]
where $F^\perp$ is any subspace of $F'$ complementary to $F$ and $A$ is a smooth interpolation map with values in $F$. The claim is verified by the observation that $(\id-A(Y,-))(F^\perp)\subseteq F'\cap\mfm_Y$ and maps isomorphically onto $F'/F$.

Since both the bundle projection $\SP_d\times G_d(F',F)\ra\SP_d\times G_d(F)$ and its section are clearly smooth we have just shown that there is a well-defined smooth structure on $\mcM_d$. In fact they are induced by fibrewise linear maps on the respective exterior powers containing the Grassmannians. Namely the bundle projection is induced by
\[\Lambda^{r'-d}F'\cong\bigoplus_{i+j=r'-d}\Lambda^iF\otimes\Lambda^jF^\perp\xlra{projection}\Lambda^{r-d}F\otimes\Lambda^sF^\perp\cong\Lambda^{r-d}F\]
and the section is, on the fibre over $Y$, induced by ``taking the wedge product'' with $(v_1-A(Y,v_1))\wedge\cdots\wedge(v_s-A(Y,v_s))$ where $(v_1,\ldots,v_s)$ is some basis of $F^\perp$.

Therefore $\mcM_d$ has a well defined smooth fibrewise algebraic structure. Finally we prove that for the canonical projection $\pi:M^d\times G_d(F)\ra\SP_d\times G_d(F)$,
\[\pi^{-1}\mcM_d\subseteq M^d\times G_d(F)\subseteq M^d\times\bbP(\Lambda^{r-d}F)\]
is a chart-wise semialgebraic subset. First of all $\mcM_d$ is a union (indexed by unordered collections of positive integers $k_1,\ldots,k_n$ with sum $d$) of subsets $\mcM_{k_1,\ldots,k_n}$ of those ideals whose spectrum has weights $k_1,\ldots,k_n$. To prove chart-wise semialgebraicity of each $\pi^{-1}\mcM_{k_1,\ldots,k_n}$ consider the following partially defined map
\begin{align}
M^{(n)}\times V_{k_1}(F)\times\cdots\times V_{k_n}(F) & \lra M^d\times G_d(F) \label{eqn_stiefel_to_grassmann}\\
(x_1,\ldots,x_n,\alpha_1,\ldots,\alpha_n) & \longmapsto ((x_1^{k_1},\ldots,x_n^{k_n}),(\lspan\alpha_1)\cap\cdots\cap(\lspan\alpha_n)) \nonumber
\end{align}
with $V_{k_i}(F)$ the Stiefel manifold of orthonormal ($r-k_i$)-frames in $F$. The map is not hard to be induced by a polynomial map
\[F^{\times(r-k_1)}\times\cdots\times F^{\times(r-k_n)}\lra\Lambda^{r-k_1}F\otimes\cdots\otimes\Lambda^{r-k_n}F\lra\Lambda^{r-d}F\]
where the first map corresponds to passing from $V_{k_i}(F)$ to $G_{k_i}(F)$ while the second (linear) to taking the intersection of the subspaces.

By Tarski-Seidenberg theorem we are left to show that there exists a chart-wise semialgebraic subset $\mcX_{k_1,\ldots,k_n}$ of the source of (\ref{eqn_stiefel_to_grassmann}) which maps onto $\pi^{-1}\mcM_{k_1,\ldots,k_n}$. The primary decomposition of ideals provides such a subset
\[\mcX_{k_1,\ldots,k_n}=\{(x_1,\ldots,x_n,\alpha_1,\ldots,\alpha_n)\ |\ \lspan\alpha_i=U_F(I_i)\textrm{ with }\spec(I_i)=\{x_i\}\}\]
whose chart-wise semialgebraicity is verified by the following proposition.
\end{proof}

\begin{proposition}
The subset $\mcX_{k_1,\ldots,k_n}$ is chart-wise semialgebraic.
\end{proposition}

\begin{proof}
Let us denote $L_i=\lspan\alpha_i$. According to Lemma~\ref{int_recognizing_the_image_of_UF} the subset $\mfX_{k_1,\ldots,k_n}$ is described by the following conditions
\begin{itemize}
\item $F\cap(\mfm_{x_i})^{k_i}\subseteq L_i$ and
\item $A(\{x_i^{k_i}\},FL_i)\subseteq L_i$.
\end{itemize}
Since the codimension of $/L_i$ in $F$ is $k_i$ the first condition is equivalent to the image of $L_i$ under $F\ra J^{k_i-1}_{x_i}(M,\bbR)$ having codimension $k_i$.

Now we write the above conditions as equations for $(\alpha_1,\ldots,\alpha_n)$. The first is seen by identifying $J^{k_i-1}(M,\bbR)$ in the chart with a trivial vector bundle $M\times\bbR^t$ and viewing the jet prolongation as a fibrewise linear map $M\times F\ra M\times\bbR^t$. The condition is then equivalent to the vanishing of certain minors of an element of $\hom(\bbR^{r-k_i},\bbR^t)$ which depends polynomially on the $x_i$ and $\alpha_i$.

For the second condition let us denote also by $\alpha_i$ the matrix consisting of coordinates of the basis vectors in a fixed orthonormal basis of $F$. Then $\alpha_i\alpha_i^T$ is the orthogonal projection onto $L_i=\lspan\alpha_i$. Writing $\alpha_i=(f_{i,1},\ldots,f_{i,r-k_i})$ we may for $p\in F$ rewrite the second condition as $(\id_F-\alpha_i\alpha_i^T)A(\{x_i\}^{k_i},pf_{i,j})=0$.
\end{proof}

\begin{lemma} \label{int_recognizing_the_image_of_UF}
Let $F\subseteq\mcR$ be transverse to $\mfm_Y$ and $L\in G_d(F)$. Then $L=U_F(I)$ for some $I\in\mcM_d$ with $\mfm_Y\subseteq I$ if and only if
\begin{itemize}
\item $F\cap\mfm_Y\subseteq L$ and
\item $A(Y,FL)\subseteq L$
\end{itemize}
in which case $I=\mfm_Y+L$.
\end{lemma}

\begin{proof}
First we observe that for any $L\in G_d(F)$ we have
\begin{equation} \label{int_codimensions}
\codim_\mcR(\mfm_Y+L)=\codim_F(F\cap\mfm_Y+L)\leq d.
\end{equation}
Now if $L=U_F(I)=F\cap I$ then $\mfm_Y+L\subseteq I$ forcing $\codim_\mcR(\mfm_Y+L)=d$ and by (\ref{int_codimensions}), $F\cap\mfm_Y\subseteq L$. Second condition follows by decomposing each $f\in FL$ as
\[f=(f-A(Y,f))+A(Y,f)\in\mfm_Y+F\]
Since $f$ lies in $I$ and $\mfm_Y\subseteq I$ this implies $A(Y,f)\in F\cap I=L$.

If on the other hand the two conditions are satisfied we set $I=\mfm_Y+L$. By (\ref{int_codimensions}) the first condition implies $\codim_\mcR I=d$. It remains to show that $I$ is an ideal. Since
\[\mcR I=(\mfm_Y+F)(\mfm_Y+L)\subseteq \mfm_Y+FL\]
we are left to show that $FL\subseteq I=\mfm_Y+L$ which is guaranteed by the second condition.
\end{proof}

\begin{proof}[Proof of Theorem~\ref{int_independence_of_topology}, the properness of $\wspec$]
We are to show that $\bar U_F\mcM_d$ is closed in $\SP_d\times G_d(F)$. Therefore let $(Y_p,L_p)$ be a sequence in $\bar U_F\mcM_d$ converging to $(Y,L)\in\SP_d\times G_d(F)$. Thus there are ideals $I_p$ with $\wspec(I_p)=Y_p$ and $L_p=F\cap I_p$. We will now construct an ideal $I$ such that $L=F\cap I$. This is an easy application of Lemma~\ref{int_recognizing_the_image_of_UF} to $L$ and $Y$. The first condition is the content of Corollary~\ref{int_convergence_in_Sd}. To verify the second condition we choose for $f\in L$ a sequence $f_p\in L_p$ converging to $L$ and then
\[A(Y,gf)=\lim_p A(Y_p,gf_p)\]
lies in $L$ since $A(Y_p,gf_p)\in A(Y_p,PL_p)\subseteq L_p$ and $L_p$ converges to $L$.

The harder part is to show that indeed $\wspec(I)=Y=\{y_1^{k_1},\ldots,y_n^{k_n}\}$.
The idea is to use the primary decomposition. We choose some disjoint closed neighborhoods $C_i$ of $y_i$ and decompose
\[I_p=(I_p+\mfn_{C_1})\cap\cdots\cap(I_p+\mfn_{C_n})\]
By the convergence $\wspec(I_p)\rightarrow Y$ we easily see that
the codimension of each $I_p^i:=I_p+\mfn_{C_i}$ is $k_i$, at
least for all big enough $p$. Assuming that each sequence $F\cap
I_p^i$ converges in $G_{k_i}(F)$ to $F\cap I^i$ we certainly have
\begin{align*}
F\cap I & =\lim(F\cap I_p^1)\cap\cdots\cap(F\cap I_p^n)\subseteq \\
& \subseteq(F\cap I^1)\cap\cdots\cap(F\cap I^n)= F\cap(I^1\cap\cdots\cap I^n)
\end{align*}
and by $F\pitchfork\mcP_{d+1}$ and Proposition~\ref{int_injectivity_proposition} we conclude that $I\subseteq I^1\cap\cdots\cap I^n$. By the above and Lemma~\ref{int_recognizing_the_image_of_UF} we have $I^i=(\mfm_{y_i})^{k_i}+F\cap I^i$ and thus its spectrum consists just of $y_i$. Hence the $I^i$ are in general position and therefore the codimension of $I^1\cap\cdots\cap I^n$ is also $d$, $I=I^1\cap\cdots\cap I^n$ is the primary decomposition of $I$ proving the claim about its weighted spectrum.

The map $\mcM_d\times\mcR\ra\mcE_d$ of vector bundles can be defined alternatively as
\[\mcM_d\times\mcR\xlra{(\id_{\mcM_d},\wspec)\times\id_\mcR}\mcM_d\times\SP_d\times\mcR
\xlra{\id_{\mcM_d}\times A}\mcM_d\times F\longrightarrow\mcE_d\]
using the continuous interpolation map $A$ of Theorem~\ref{int_continuous_interpolation}. Using an inner product on $F$ one can easily find a section proving that it is indeed a quotient map.
\end{proof}

\begin{remark}
Let us denote $\mcM_d(M)=\mcM_d$. For an open subset $U\subseteq M$ we define $\mcM_d(U)\subseteq\mcM_d(M)$ by the pullback square
 \[\xymatrix{
  \mcM_d(U) \ar@{c->}[r] \ar[d] & \mcM_d(M) \ar[d]^{\wspec}
 \\
  \SP_d(U) \ar@{c->}[r] & \SP_d(M)
 }\]
In other words $\mcM_d(U)$ is the space of ideals in $\mcR$ whose spectra lie in $U$. Theorem~\ref{int_independence_of_topology} implies that the topology of $\mcM_d(U)$ is independent of $M$ and $\mcM_d(M)$ is a union of $\mcM_d(U)$ as $U$ varies over coordinate charts allowing one to build $\mcM_d(M)$ from its local versions $\mcM_d(U)$. On the other hand the glueing maps are not affine and so \cite{Glaeser} could not be used to glue $\mcM_d$ from the local versions.
\end{remark}

\section{The structure of $\mcM_d$}

We will now prove some structure theorems explaining what the spaces $\mcM_d$ look like. Concrete examples for low codimensions $d$ can be found in \cite{Thesis}. First we prove a useful proposition which enables recognizing whether maps into $\mcM_d$ are continuous, smooth or even immersions.

\begin{proposition} \label{int_smooth_maps_into_Md}
Let $N$ be a topological space, let $F\pitchfork\mcM_d$ and let
 \[\varphi:N\times\mcR\ra\bbR^d\]
be a not necessarily continuous mapping such that for each $x\in N$ the partial map $\varphi(x,-):\mcR\ra\bbR^d$ is surjective linear with kernel $I_x=\ker\varphi(x,-)$ an ideal in $\mcR$. Denote by $\psi$ the map $\psi:N\ra\mcM_d\subseteq\SP_d\times G_d(F)$ sending $x\in N$ to $I_x$ and by $\psi_2$ its second coordinate,~i.e.~$\psi_2(x)=F\cap I_x$.
\begin{enumerate}
\item[a)]
If for each $f\in\mcR$ the partial map $\varphi(-,f):N\ra\bbR^d$ is continuous then so is $\psi_2$.
\item[b)]
If $N$ is a smooth manifold and all $\varphi(-,f)$ are smooth then so is $\psi_2$.
\item[c)]
If in addition $F\pitchfork\mcM_{d+1}$ then $X\in T_xN$ lies in the kernel of the differential $(\psi_2)_*:TN\ra TG_d(F)$ if and only if
 \begin{equation} \label{int_condition_on_ideals}
  I_x\subseteq\{f\in\mcR\ |\ d(\varphi(-,f))(X)=0\}
 \end{equation}
\end{enumerate}
\end{proposition}

\begin{remark}
It is not hard to prove continuity of the first component $\psi_1(x)=\wspec(I_x)$ in a). It seems regretable that we have not been able to prove smoothness of $\psi_1$ in~b). In our applications however this will be automatic.

For $M$ one-dimensional $\psi_1$ is smooth: locally in a neighbourhood of $x_0\in M=\bbR$ and $\wspec(I_{x_0})=\{y^k\}$ one can choose a (unique) normed polynomial generator of $I_x$ of degree exactly $k$ in a smooth manner. Its coefficients constitute generators (over $C^\infty(\bbR,\bbR)$ acting by composition from the left) for the ring of smooth functions on $\SP_k$ defined in a neighbourhood of $\{y^k\}$ and pull back to smooth functions on $N$.
\end{remark}

\begin{proof}
We skip the proof of a) as it will become clear from the proof of b).

Let $x\in N$ and denote $L=\psi_2(x)=F\cap I_x$. We
choose a complementary subspace $L^\perp$ to $L$ inside $F$ and get a
chart on $G_d(F)$
 \[\Hom(L,L^\perp)\lra G_d(F)\]
given by sending a map $\alpha$ to its graph inside $L\times
L^\perp\cong F$. Let us consider the restriction $\varphi_F$ of
$\varphi$ to $N\times F$ and write it in the form
 \[\varphi_F:N\times L\times L^\perp\lra\bbR^d\]
Observe that the differential $d\varphi_F|_{L^\perp}$ is an isomorphism of
$L^\perp$ on $\bbR^d$ near $\{x\}\times L\times L^\perp$. In particular there
is a unique solution $\alpha(y)(v)$ to the equation
 \[\varphi_F(y,v,\alpha(y)(v))=0\]
and it is automatically smooth. Clearly $\alpha(y):L\ra L^\perp$ is the
expression of $\psi_2(y)$ in the above coordinate chart (with
$\alpha(x)=0$).

Starting the proof of c) we have a formula for the derivative
 \[d\alpha(X)(v)=(d(\varphi_F(x,v,-)))^{-1}(d(\varphi_F(-,v,0)))(X)\]
In particular $X\in\ker(\psi_2)_*$ if and only if for each $v\in L$ it
lies in the kernel of $d(\varphi_F(-,v,0))$. To explain this
condition we introduce
 \[I_X:=\{f\in\mcR\ |\ \varphi(x,f)=0,\ d(\varphi(-,f))(X)=0\}\]
As the name suggests it is an ideal and to prove this one observes
that for each $y\in N$ we have a multiplication on $\bbR^d$
arising from the identification $\mcR/I_y\cong\bbR^d$. This family
is smooth in the sense of the map
 \[\mu:N\lra\Hom(\bbR^d\otimes\bbR^d,\bbR^d)\]
being smooth. If we temporarily denote
 \[f(x):=\varphi(x,f)\quad\textrm{and}\quad
 df(X):=d(\varphi(-,f))(X)\]
then for $f,g\in\mcR$ we get
 \[d(fg)(X)=\mu(x)\bigl(f(x)\otimes dg(X)+df(X)\otimes g(x)\bigr)+d\mu(X)(f(x)\otimes g(x))\]
Therefore if one of $f$, $g$ lies in $I_X$ then so does their
product.

The condition (\ref{int_condition_on_ideals}) from the statement
is then equivalent to $I_X=I_x$. Assuming that this equality
holds, every $v\in L\subseteq I_x$ lies in $I_X$ implying that
$d(\varphi_F(-,v,0))(X)=0$. Therefore in this case $(\psi_2)_*(X)=0$.
If on the other hand $I_X\subsetneqq I_x$ then there is an ideal
$J$ which is maximal among those for which $I_X\subseteq
J\subsetneqq I_x$. Necessarily $J\in\mcM_{d+1}$ and by our transversality
assumption $F\cap I_X\subseteq F\cap J\subsetneqq F\cap I_x=L$ so that there is $v\in
L$ for which $v\not\in I_X$ implying that
$d(\varphi_F(-,v,0))(X)\neq0$ and $\psi_*(X)\neq0$.
\end{proof}

The space $\mcM_d$ contains as a subspace the configuration
space
 \[M^{[d]}=M^{(d)}/\Sigma_d\subseteq\SP_d\]
We will show now that it is in general an embedded submanifold.
An interesting question is whether $\mcM_d$ is the closure of $M^{(d)}$.

\begin{proposition} \label{int_configuration_space_inclusion}
Let $F$ be a finite dimensional linear subspace of $\mcR$ such
that $F\pitchfork\mcM_{d+1}$. Then the inclusion
\begin{align*}
\psi:M^{[d]}\subseteq\SP_d & \lra\mcM_d\subseteq\SP_d\times G_d(F) \\
Y & \longmapsto\mfm_Y
\end{align*}
is a smooth embedding. In fact $\psi_2$ already is.
\end{proposition}

\begin{proof}
As the map in question has a smooth inverse we only need to
show that it is an immersion. First we express $\psi$ locally via
a map $\varphi:M^{[d]}\times\mcR\ra\bbR^d$ as in Proposition
\ref{int_smooth_maps_into_Md} and compute the kernel of $(\psi_2)_*$
using the same proposition. Therefore let $(x_1,\ldots,x_d)\in M^{(d)}$
and identify a neighbourhood of $[(x_1,\ldots,x_d)]\in M^{[d]}$
with a product $U_1\times\cdots\times U_d$ of disjoint
neighbourhoods $U_i$ of $x_i$. Then we can define
$\varphi:U_1\times\cdots\times U_d\times\mcR\ra\bbR^d$ by
 \[(z_1,\ldots,z_d,f)\mapsto(f(z_1),\ldots,f(z_d))\]
Clearly all the assumptions of Proposition \ref{int_smooth_maps_into_Md}
are satisfied and so $\psi_2$ is a smooth map as is $\psi_1=\id$. Also for
 \[(X_1,\ldots,X_d)\in T_{x_1}U_1\times\cdots\times T_{x_d}U_d\]
we have $d(\varphi(-,f))(X_1,\ldots,X_d)=(df(X_1),\ldots,df(X_d))$
and this can be zero on $\mfm_{\{x_1,\ldots,x_d\}}$ only if
$X_1=\cdots=X_d=0$.
\end{proof}

We will now generalize the previous proposition. Namely we describe certain subsets of $\mcM_d(M)$ of ideals of a ``fixed type''. They are injectively immersed submanifolds. Also every ideal has some (unique) type and so $\mcM_d$ is in fact a disjoint union (over all possible types) of these submanifolds.

We take the following construction from section 35 of \cite{KMS}.
A \emph{Weil algebra} is a finite dimensional associative,
commutative algebra $A$ over $\bbR$ with a unit such that
$A=\bbR\oplus N$ where $N$ is the ideal of nilpotent elements.
Equivalently it could be described as a quotient algebra of
$J^r_0(\bbR^m,\bbR)$ for some $r$ and $m$. Therefore we can get
all ideals of $\mcR$ with a spectrum consisting of a single point
as kernels of surjective algebra homomorphisms $\mcR\ra A$ for
some Weil algebra $A$ (namely $A$ is the quotient of $\mcR$ by
that ideal). We give the set of all such homomorphisms (surjective
or not) a smooth structure in such a way that the map sending such
a homomorphism to the spectrum of its kernel is a bundle
projection. This bundle is called the \emph{Weil bundle}
associated to $A$.

We first give a construction of this bundle and then show that its
points can be indeed identified with homomorphism $\mcR\ra A$. We
start with the restriction
 \begin{equation} \label{int_jet_bundle}
  J^r_{0,\mathrm{diff}}(\bbR^m,M)\ra M
 \end{equation}
of the jet bundle $J^r(\bbR^m,M)\ra M$ to the subspace of all
invertible jets with source~$0$. Setting
$G^r_m:=J^r_{0,\mathrm{diff}}(\bbR^m,\bbR^m)_0$, the Lie group of
all invertible jets with source and target~$0$, we see that
(\ref{int_jet_bundle}) is a principal $G^r_m$-bundle and so we can
define
 \begin{equation} \label{int_definition_of_TAM}
  T_AM:=J^r_{0,\mathrm{diff}}(\bbR^m,M)\times_{G^r_m}\Hom_{\mathrm{alg}}(J^r_0(\bbR^m,\bbR),A)
 \end{equation}
This clearly expresses $T_AM$ as a smooth bundle over $M$ with
fibre
 \[\Hom_{\mathrm{alg}}(J^r_0(\bbR^m,\bbR),A)\cong N^m\]
Moreover we have a bijection defined in terms of
(\ref{int_definition_of_TAM}) by the formula
\begin{align*}
T_AM & \xlra{\cong}\Hom_{\mathrm{alg}}(\mcR,A) \\
[j^r_xg,\varphi] & \longmapsto\left(\mcR\xra{j^r_x} J^r_x(M,\bbR)\xra{g^*}J^r_0(\bbR^m,\bbR)\xra{\varphi} A\right)
\end{align*}
It is easily seen to be a bijective correspondence (that a kernel
of any $\mcR\ra A$ has a spectrum consisting of only a single
point follows from the fact that in $A$, the only idempotents are
$0$ and $1$). We have a subbundle
 \[
  \check{T}_AM:=J^r_{0,\mathrm{diff}}(\bbR^m,M)\times_{G^r_m}\SurHom_{\mathrm{alg}}(J^r_0(\bbR^m,\bbR),A)
 \]
which then corresponds to surjective algebra homomorphisms
$\mcR\ra A$.

One says that an ideal $I$ is of type $A$ if $\mcR/I\cong A$. An
ideal $I$ of type $A$ can then be identified with a class of
surjective homomorphisms $\mcR\ra A$, namely with the class of all
those homomorphisms that have kernel $I$. In this way we get a
space $J^AM$ of all ideals of type $A$ as a certain quotient of
$\check{T}_AM$. A crucial observation in \cite{Alonso} is that the
action of $G^r_m$ on
$\SurHom_{\mathrm{alg}}(J^r_0(\bbR^m,\bbR),A)$ is transitive and
so, after choosing some
$\alpha_0\in\SurHom_{\mathrm{alg}}(J^r_0(\bbR^m,\bbR),A)$, one can
identify it with the quotient of $G^r_m$ by the stabilizer of
$\alpha_0$. In the same terminology an ideal in
$J^r_0(\bbR^m,\bbR)$ of type $A$ is a class of $G^r_m$ modulo the
stabilizer of $\ker\alpha_0$. Therefore the space of ideals of
type $A$ can be identified with the smooth bundle
 \[
  J^AM\cong J^r_{0,\mathrm{diff}}(\bbR^m,M)/\mathrm{St}(\ker\alpha_0)\lra M
 \]
and clearly the smooth structure does not depend on the choice of
$\alpha_0$.

\begin{proposition} \label{int_Weil_bundle_inclusion}
Let $F\pitchfork\mcM_{d+1}$ be a finite dimensional linear
subspace of $\mcR$ and let $A$ be a $d$-dimensional Weil algebra.
Then the inclusion
 \[\iota:J^AM\subseteq \mcM_d\subseteq\SP_d\times G_d(F)\]
is an injective immersion.
\end{proposition}

\begin{proof}
To be on the safe side we prove that in fact the second component $\iota_2$ of $\iota$ is an injective immersion. This has the advantage of staying solely in the world of smooth manifolds where the meaning of an ``immersion'' is clear.

This is another application of Proposition
\ref{int_smooth_maps_into_Md}. Consider the map
\begin{align*}
\varphi:J^r_{0,\mathrm{diff}}(\bbR^m,M)\times\mcR & \lra A \\
(j^r_x(g),f) & \longmapsto\alpha_0(j^r_x(fg))
\end{align*}
Clearly this map is smooth in the first and surjective linear in the second variable and hence in the sense of Proposition~\ref{int_smooth_maps_into_Md} it defines
 \[\psi:J^r_{0,\mathrm{diff}}(\bbR^m,M)\ra\mcM_d\subseteq\SP_d\times G_d(F)\]
which is also smooth (the first component being the bundle projection for $J^AM$). As we have a commutative diagram
 \[\xymatrix{
  J^r_{0,\mathrm{diff}}(\bbR^m,M) \ar[r]^-{\psi_2} \ar@{->>}[d] &
  G_d(F)
 \\
  J^r_{0,\mathrm{diff}}(\bbR^m,M)/\mathrm{St}(\ker\alpha_0)
  \ar@{-->}[ru]_{\iota_2}
 }\]
in order to show that the dashed arrow $\iota_2$ is an immersion we need to
identify $\ker(\psi_2)_*$. Proposition \ref{int_smooth_maps_into_Md} gives
an answer in terms of the kernel of the differential (say at
$j^r_xg$) of the map $\varphi(f,-)$ which can be decomposed as
 \[J^r_{0,\mathrm{diff}}(\bbR^m,M)\xlra{f_*}
 J^r_0(\bbR^m,\bbR)\xlra{\alpha_0} A\]
To give a tangent vector in
$T_{j^r_0g}J^r_{0,\mathrm{diff}}(\bbR^m,M)$ is the same as to give
an element of $T_{\id}J^r_{0,\mathrm{diff}}(\bbR^m,\bbR^m)$ and then
compose with $g$. The elements of
$T_{\id}J^r_{0,\mathrm{diff}}(\bbR^m,\bbR^m)$ arise from vector
fields. Therefore let $X:\bbR^m\ra T\bbR^m$ be a local vector field
with a local flow
 \[\gamma:\bbR^m\times\bbR\lra \bbR^m\]
Under our identifications it defines a tangent vector
 \[\hat{X}:=\left.\frac{d}{dt}\right|_{t=0}j^r_0(g\gamma(-,t))
 \in T_{j^r_0g}J^r_{0,\mathrm{diff}}(\bbR^m,M)\]
Then for each $f\in\mcR$ we get
 \[d(\alpha_0f_*)(\hat{X})=
 \left.\frac{d}{dt}\right|_{t=0}\alpha_0(j^r_0(fg\gamma(-,t)))=
 \alpha_0(j^r_0(X(fg)))\]
Suppose that $X(0)\neq 0$. Then we claim that there exists an
$f\in\psi_2(j^r_xg)$ for which this expression is nonzero as well. In
other words such $\hat{X}$ can never lie in $\ker(\psi_2)_*$.\footnote{Also, and this is obvious, it can never lie in $\ker(\psi_1)_*$.} We postpone the proof of this claim and thus assume that the only
$\hat{X}$ which could produce an element in this kernel are the
vectors tangent to the submanifold
$J^r_{0,\mathrm{diff}}(\bbR^m,M)_x$. The Lie group $G^r_m$ acts
simply transitively on this space. Let $Y$ be an element of the Lie
algebra of $G^r_m$. Then we obtain a vector field (with $p$ running
over $J^r_{0,\mathrm{diff}}(\bbR^m,M)_x$)
 \[Y^+(p):=\left.\frac{d}{dt}\right|_{t=0}p\cdot\exp(tY)\]
and we also have similar vector fields on $J^r_0(\bbR^m,\bbR)$. The
restriction of $\varphi(f,-)$ is simply the composition
 \[J^r_{0,\mathrm{diff}}(\bbR^m,M)_x\xlra{f_*}J^r_0(\bbR^m,\bbR)
 \xlra{projection}J^r_0(\bbR^m,\bbR)/\ker\alpha_0\cong A\]
Therefore $Y^+(j^r_0g)$ lies in $\ker(\psi_2)_*$ if and only if for each
$f\in\psi_2(j^r_0g)$ we have $df_*(Y^+(j^r_0g))\in\ker\alpha_0$. As
the map $f_*$ is $G^r_m$-equivariant we can rewrite
 \[df_*(Y^+(j^r_0g))=Y^+(f_*j^r_0g)=Y^+(j^r_0(fg))\]
for the corresponding canonical vector field on
$J^r_0(\bbR^m,\bbR)$. Now observe that we get all possible values
$j^r_0(fg)\in\ker\alpha_0$ by varying $f$ over $\psi_2(j^r_0g)$.
Therefore $Y^+(j^r_0g)\in\ker(\psi_2)_*$ if and only if
$Y^+(\ker\alpha_0)\subseteq\ker\alpha_0$. These $Y$ clearly
constitute the Lie algebra of the stabilizer
$\mathrm{St}(\ker\alpha_0)$ and therefore we conclude that
$\ker(\psi_2)_*$ is exactly the vertical tangent bundle of
 \[J^r_{0,\mathrm{diff}}(\bbR^m,M)\lra
 J^r_{0,\mathrm{diff}}(\bbR^m,M)/\mathrm{St}(\ker\alpha_0)\cong J^AM\]
Consequently $\psi_2$ induces on the quotient $J^AM\cong
J^r_{0,\mathrm{diff}}(\bbR^m,M)/\mathrm{St}(\ker\varphi_0)$ an
immersion $\iota_2:J^AM\lra G_d(F)$.

Now we prove the remaining claim. Because we assume that $X(0)\neq
0$ we can find a local diffeomorphism $h:(\bbR^m,0)\ra(\bbR^m,0)$
such that $h^*X=\partial_{x_1}$. Then
 \[X(fg)=\partial_{x_1}(fgh)\circ h^{-1}\]
We denote by $K$ the kernel of
 \[J^r_0(\bbR^m,\bbR)\xlra{(h^{-1})^*}J^r_0(\bbR^m,\bbR)\xlra{\varphi_0}A\]
and we are looking for $f\in\mcR$ such that $j^r_0(fgh)\in K$ but
$j^r_0(\partial_{x_1}(fgh))\not\in K$. Let
$K\subseteq(x_1^k,x_2,\ldots,x_m)$ but
$K\not\subseteq(x_1^{k+1},x_2,\ldots,x_m)$. Both $g$ and $h$ being
diffeomorphisms there exists $f\in\mcR$ such that $j^r_0(fgh)\in
K-(x_1^{k+1},x_2,\ldots,x_m)$. Then $fgh=x_1^k\cdot\lambda$ modulo
$(x_2,\ldots,x_m)$ with $\lambda(0)\neq 0$ and it is easy to see
that $\partial_{x_1}(fgh)=x_1^{k-1}\mu$ modulo $(x_2,\ldots,x_m)$
with $\mu(0)\neq 0$ so that $j^r_0(\partial_{x_1}(fgh))$ does not
lie in $(x_1^k,x_2,\ldots,x_m)$ and in particular it does not lie in
$K$.
\end{proof}

\begin{question}
Is the inclusion $\iota:J^AM\hookrightarrow\mcM_d$ an
embedding? Quite easily (reducing to a local question and using
polynomials) one can reduce this problem to the question of the
canonical map
 \[G^r_m/(G^r_m\cap\mathrm{St}(\ker\alpha_0))\hookrightarrow Gl(J^r_0(\bbR^m,\bbR))/\mathrm{St}(\ker\alpha_0)\]
being an embedding.
\end{question}

An easy generalization to the case of finitely many Weil algebras
$A_i$, $i=1,\ldots,n$, gives a bundle
 \[
  \check{T}_{A_1,\ldots,A_n}M=
  \left.\left(\check{T}_{A_1}M\times\cdots\times\check{T}_{A_n}M\right)\right|_{M^{(n)}}\lra
  M^{(n)}
 \]
togeter with a bijection
$\check{T}_{A_1,\ldots,A_n}M\cong\SurHom_{\mathrm{alg}}(\mcR,A_1\times\cdots\times
A_n)$. Every ideal $I\in\mcM_d$ of type $A_1\times\cdots\times
A_n$ can be clearly recovered as a kernel of such surjective
homomorphism. In this way we obtain a space $J^{A_1,\ldots,A_n}M$ of
ideals of a fixed type $A_1\times\cdots\times A_n$ as a quotient of
$\check{T}_{A_1,\ldots,A_n}M$. Moreover we can again identify
$J^{A_1,\ldots,A_n}M$ with a quotient\footnote{Here
$(J^r_{0,\mathrm{diff}}(\bbR^m,M))^{(n)}_{\ushort{M}}$ denotes the
restriction of the power $(J^r_{0,\mathrm{diff}}(\bbR^m,M))^n\ra
M^n$ of the jet bundle to the subspace $M^{(n)}\subseteq M^n$. In
particular as the original bundle was a principal $G^r_m$-bundle the
resulting bundle over $M^{(n)}$ will be a principal
$(G^r_m)^n$-bundle on which there is an action of the symmetric
group $\Sigma_n$. Composing with the quotient map $M^{(d)}\ra M^{[d]}=M^{(d)}/\Sigma_d$ by this action one gets a
principal $(G^r_m\wr\Sigma_n)$-bundle
$(J^r_{0,\mathrm{diff}}(\bbR^m,M))^{(n)}_{\ushort{M}}\ra M^{[n]}$.}
 \[(J^r_{0,\mathrm{diff}}(\bbR^m,M))^{(n)}_{\ushort{M}}
 \times_{G^r_m\wr\Sigma_n}(G^r_m\wr\Sigma_n)/\mathrm{St}(\ker\alpha_0)\]
for some (any) surjective homomorphism
$\alpha_0:(J^r_0(\bbR^m,\bbR))^n\ra A_1\times\cdots\times A_n$. The
canonical inclusion map
$\iota:J^{A_1,\ldots,A_n}M\hookrightarrow\mcM_d$ with $d=\dim
A_1+\cdots+\dim A_n$ is an injective immersion. For a proof observe
that locally $(J^r_{0,\mathrm{diff}}(\bbR^m,M))^{(n)}_{\ushort{M}}$
is just a product of $J^r_{0,\mathrm{diff}}(\bbR^m,M)$ and so one
can almost copy the proof of Proposition~\ref{int_Weil_bundle_inclusion} (also see the proof of Proposition~\ref{int_configuration_space_inclusion}).

\appendix

\section{Interpolation}

In this section we construct a continuous interpolation and thus prove Theorem~\ref{int_continuous_interpolation}. The main ingredient is integration over simplices in $\bbR^m$.

Let $f:\bbR^m\rightarrow\bbR$ be a smooth function. If
$(x_0,\ldots,x_r)\in (\bbR^m)^{r+1}$ we denote by
$[x_0,\ldots,x_r]:\Delta^r\rightarrow\bbR^m$ the unique affine map
sending the vertices of the standard $r$-simplex $\Delta^r$ to
$x_0,\ldots,x_r$. It is obvious that this gives a bijective
correspondence between $(\bbR^m)^{r+1}$ and affine maps
$\Delta^r\rightarrow\bbR^m$. We will denote a general affine map
$\Delta^r\rightarrow\bbR^m$ by $\sigma$, if we do not want to
emphasize the values at vertices. By embedding it linearly into
$\bbR^r$ we give $\Delta^r$ the Lebesgue measure in which the volume
is $1$. Then we are able to define unambiguously $\mcI(f,\sigma)\in\Hom(S^r\bbR^m,\bbR)$ a symmetric $r$-form on $\bbR^m$ to be
\[\mcI(f,\sigma)=\int_{\Delta^r}f^{(r)}\sigma\]
where $f^{(r)}:\bbR^m\ra\Hom(S^r\bbR^m,\bbR)$ denotes the $r$-fold
derivative of $f$. In an obvious way  by omitting $x_i$ we get
$\partial_i\sigma:\Delta^{r-1}\rightarrow\bbR^m$ and thus also
\[\mcI(f,\partial_i\sigma)\in\Hom(S^{r-1}\bbR^m,\bbR)\]

\begin{lemma}\label{int_integration_formula}
For any smooth function $g:\bbR^m\rightarrow\bbR$ and for any
$\sigma=[x_0,\ldots,x_r]$, the following holds for $0\leq i,j\leq
r$
\[\int_{\Delta^r}g'_{x_j-x_i}\sigma=
r\left(\int_{\Delta^{r-1}}g(\partial_i\sigma)-
\int_{\Delta^{r-1}}g(\partial_j\sigma)\right)\] where $g'_{x_j-x_i}$
denotes the derivative of $g$ in the direction $x_j-x_i$. In
particular by taking $g=f^{(r-1)}$ we have
\[\mcI(f,\sigma)(v_1,\ldots,v_{r-1},x_j-x_i)=
r(\mcI(f,\partial_i\sigma)-
\mcI(f,\partial_j\sigma))(v_1,\ldots,v_{r-1})\]
\end{lemma}

\begin{proof}
Define a map $\delta:\Delta^{r-2}\rightarrow\Delta^r$ to be
\[[e_0,\ldots,\hat{e}_i,\ldots,\hat{e}_j,\ldots,e_r]\]
where $e_n$ are the vertices of $\Delta^r$. We think of
$\Delta^{r-1}$ as a convex span of $\Delta^{r-2}$ (with vertices
$e_1,\ldots,e_{r-1}$) and an additional point $e_0$. Then we can
define a homotopy
\[h:\Delta^{r-1}\times I\rightarrow\Delta^r\]
by a formula (with $x$ running over $\Delta^{r-2}$)
\[h(t_0e_0+(1-t_0)x,t)=t_0((1-t)e_i+te_j)+(1-t_0)\delta(x)\]
One sees easily that $h(-,0)$ is the inclusion of the $j$-th face of
$\Delta^r$ and $h(-,1)$ the inclusion of the $i$-th one. To compute
the determinant of $h'$ we choose a basis
$(e_1-e_0,\ldots,e_{r-1}-e_0,e)$ of $T(\Delta^{r-1}\times I)$ where
$e$ is the unit tangent vector of $I$. Then
\[h'(t_0e_0+(1-t_0)x,t)(e_n-e_0)=\delta(e_n)-e_i-t(e_j-e_i)\]
and
\[h'(t_0e_0+(1-t_0)x,t)(e)=t_0(e_j-e_i)\]
Hence we easily get a formula
\[|\det h'(t_0e_0+(1-t_0)x,t)|=ct_0\]
for some constant $c$ and it is not difficult to see that $c=r$.
Then
\[\int_{\Delta^r}g'_{x_j-x_i}\sigma=
\int_{\Delta^{r-1}\times I}rt_0g'_{x_j-x_i}\sigma h=
r\int_{\Delta^{r-1}}\int_0^1t_0g'_{x_j-x_i}\sigma h(-,t)\
\textrm{d}t\] Now note that
\[\tfrac{\partial}{\partial t}(g\sigma h)=t_0g'_{x_j-x_i}\sigma h\]
and so
\begin{align*}
\int_{\Delta^{r-1}}\int_0^1t_0g'_{x_j-x_i}\sigma h(-,t)\textrm{d}t & =\int_{\Delta^{r-1}}g\sigma h(-,1)-\int_{\Delta^{r-1}}g\sigma h(-,0) \\
& =\int_{\Delta^{r-1}}g(\partial_i\sigma)-\int_{\Delta^{r-1}}g(\partial_j\sigma)
\end{align*}
\end{proof}

\begin{corollary} \label{int_taylor_expansion}
The following formula holds
 \begin{equation} \label{int_taylor_expansion_expression}
  \begin{split}
   f(x_r)=\mcI(f,[x_0])
   +\cdots & +\tfrac{1}{i!}\cdot\mcI(f,[x_0,\ldots,x_i])(x_r-x_0,\ldots,x_r-x_{i-1})+\cdots
  \\
   & +\tfrac{1}{r!}\cdot\mcI(f,[x_0,\ldots,x_r])(x_r-x_0,\ldots,x_r-x_{r-1})
  \end{split}
 \end{equation}
\end{corollary}

\begin{proof}
We use induction with respect to $r$. According to the previous
lemma
 \begin{equation} \label{int_highest_term_in_taylor_expansion}
  \tfrac{1}{r!}\cdot\mcI(f,[x_0,\ldots,x_r])(x_r-x_0,\ldots,x_r-x_{r-1})
 \end{equation}
is equal to
 \[\tfrac{1}{(r-1)!}\cdot(\mcI(f,[x_0,\ldots,\widehat{x_{r-1}},x_r])-\mcI(f,[x_0,\ldots,x_{r-1}]))(x_r-x_0,\ldots,x_r-x_{r-2})\]
Adding the remaining terms of
(\ref{int_taylor_expansion_expression}) to
(\ref{int_highest_term_in_taylor_expansion}) and using the inductive
hypothesis on $[x_0,\ldots,\widehat{x_{r-1}},x_r]$ we prove the
inductive step.
\end{proof}

\begin{corollary} \label{int_equivalence}
The following conditions are equivalent
\begin{itemize}
\item[$(i)$]{$\mcI(f,\tau)=0$ for all the faces $\tau$ of $\sigma$.}
\item[$(ii)$]{$\mcI(f,\partial_1\cdots\partial_r\sigma)=\cdots=
\mcI(f,\partial_i\cdots\partial_r\sigma)=\cdots=\mcI(f,\partial_r\sigma)=\mcI(f,\sigma)=0$.}
\end{itemize}
\end{corollary}

\begin{proof}
We assume $(ii)$. By induction we can also assume that for all the
faces $\tau$ of $\partial_r\sigma$, we have $\mcI(f,\tau)=0$. By
the previous lemma $\mcI(f,\partial_i\sigma)=0$ for all $i$. As
there is a common face of $\partial_i\sigma$ and
$\partial_r\sigma$ we see that up to a renumbering of vertices the
condition $(ii)$ is satisfied for $\partial_i\sigma$ and by
induction again we get $(i)$ for all the faces of
$\partial_i\sigma$.
\end{proof}

\begin{corollary} \label{int_vanishing_of_integrals_implies_vanishing_of_jets}
Let $\sigma=[x_0,\ldots,x_r]$. If
\[\mcI(f,\partial_1\cdots\partial_r\sigma)=\cdots=
\mcI(f,\partial_i\cdots\partial_r\sigma)=\cdots=\mcI(f,\partial_r\sigma)=\mcI(f,\sigma)=0\]
and if $y$ appears $k$-times in $x_0,\ldots,x_r$ then
$j^{k-1}_yf=0$.
\qed
\end{corollary}

\begin{remark}
The converse is not true in general (with the exception
$x_0=\cdots=x_r$) for dimensional reasons unless $m=1$: if
$\{y_1^{k_1},\ldots,y_n^{k_n}\}$ denotes the class of
$(x_0,\ldots,x_r)$ in $\SP_{r+1}(\bbR)=\bbR^{r+1}/\Sigma_{r+1}$
then the vanishing of $j^{k_i-1}_{y_i}f$ for all $i=1,\ldots,n$ implies
$\mcI(f,[x_0,\ldots,x_j])=0$ for all $j=0,\ldots,r$.
\end{remark}

Now we will explain how this leads to an interpolation map. First we
restrict ourselves to interpolation at points close to a single
point, later generalizing to a number of points. This is only to
lighten the notation a bit. Going back from the local situation to
the case of a smooth manifold $M$ we identify a neighborhood of $y$ with
$\bbR^m$. We also fix a complementary linear subspace $D$ to the
ideal $(\mfm_y)^k$ and define the following map
\begin{align*}
G:(\bbR^m)^k\times\mcR & \lra(\bbR^m)^k\times\Hom(S^0\bbR^m,\bbR)\times\cdots\times\Hom(S^{k-1}\bbR^m,\bbR) \\
& \qquad\qquad\cong(\bbR^m)^k\times J^{k-1}_*(\bbR^m,\bbR)
\end{align*}
(with $J^{k-1}_*(\bbR^m,\bbR)$ being any $J^{k-1}_x(\bbR^m,\bbR)$ -- they
are all identified via translations) by the formula
\[G((x_1,\ldots,x_k),f)=((x_1,\ldots,x_k),\mcI(f,[x_1]),\ldots,\mcI(f,[x_1,\ldots,y_x]))\]
We denote by $G_D$ its restriction
\[G_D:(\bbR^m)^k\times D\rightarrow(\bbR^m)^k\times J^{k-1}_*(\bbR^m,\bbR)\]
Note that for each $f\in\mcR$ the map $G(-,f)$ is continuous (in
fact smooth) and therefore so is $G_D$. Our transversality
condition on $D$ implies that on the fibres over $(y,\ldots,y)$
\[(G_D)_{(y,\ldots,y)}:D\rightarrow J^{k-1}_*(\bbR^m,\bbR)\]
is a linear isomorphism. Hence we find a neighborhood of
$(y,\ldots,y)$ in $M^k$ of the form $U^k$ with $U$ compact convex,
such that the restriction
\[G_D:U^k\times D\rightarrow U^k\times J^{k-1}_*(\bbR^m,\bbR)\]
is an isomorphism of vector bundles over $U^k$. Hence we can
define a map
\[\hat{A}:U^k\times\mcR\xlra{G}U^k\times J^{k-1}_*(\bbR^m,\bbR)\xlra{G_D^{-1}}U^k\times D\rightarrow D\]
Now note that according to Corollary \ref{int_equivalence} the
value of $\hat{A}$ does not depend on the ordering of the points
and hence we get
\[A:U^k/\Sigma_k\times\mcR\rightarrow D\]
with the property that $A(Y,f)$ is an interpolation of $f$ at $Y$,~i.e.~such that $f\equiv A(Y,f)$ modulo $\mfm_Y$.

\begin{proposition}
The interpolation map $A$ is continuous and $A(\{y^k\},f)=0$
whenever $f\in(\mfm_y)^k$.
\end{proposition}

\begin{proof}
By the construction of $A$ it is enough to show the continuity of
each component
\[G_i:U^k\times\mcR\rightarrow\Hom(S^i\bbR^m,\bbR)\]
of $G$, $i=0,\ldots,k-1$. Hence let us fix
$((x_1,\ldots,x_k),f)\in U^k\times\mcR$ and denote
$X=(x_1,\ldots,x_k)$ for a short. We choose a norm on
$\Hom(S^i\bbR^m,\bbR)$ and a neighborhood
\[\{\alpha\in\Hom(S^i\bbR^m,\bbR)\ |\ ||\alpha-G_i(X,f)||<\varepsilon\}\]
Because of the continuity of $G_i(-,f)$ there is a neighborhood
$W$ of $X$ on which
\[||G_i(-,f)-G_i(X,f)||<\varepsilon/2\]
Hence if $g$ is such that $||g^{(i)}-f^{(i)}||<\varepsilon/2$ on
$U$ then also
\[||G_i(-,g)-G_i(-,f)||<\varepsilon/2\]
on $U$ and finally on $W$ we have
\[||G_i(-,g)-G_i(X,f)||\leq||G_i(-,g)-G_i(-,f)||+||G_i(-,f)-G_i(X,f)||<\varepsilon\]
Because the condition on $g^{(i)}$ defines a neighborhood of $f$ in
$\mcR$, this finishes the proof.
\end{proof}

\begin{example}[Kergin interpolation] \label{example_Kergin_interpolation}
Let us take for $D$ the set of all polynomials of degree less than $k$. Then over each $(x_1,\ldots,x_k)$ the matrix of $G_D$ is upper triangular with all entries polynomial in $(x_1,\ldots,x_k)$ and diagonal entries constant. Therefore $G_D^{-1}$ is polynomial and, assuming that $f$ is a polynomial, so is $G(-,f)$ and thus $\hat A(-,f)$. Also one can take $U=\bbR^m$.
\end{example}

Now if we have an arbitrary
$Y=\{y_1^{k_1},\ldots,y_n^{k_n}\}\in\SP_d$ we identify a
neighbourhood of each $y_i$ with $\bbR^m$. Writing
\[\bbR^{md}\cong(\bbR^m)^{k_1}\times\cdots\times(\bbR^m)^{k_n}\]
we replace $G$ by the corresponding map
\[G:\bbR^{md}\times\mcR\rightarrow\bbR^{md}\times J^{k_1-1}_*(\bbR^m,\bbR)\times\cdots\times J^{k_n-1}_*(\bbR^m,\bbR)\]
Again if $D$ is a complementary linear subspace to $\mfm_Y\mcR$ then on the fibres
\[(G_D)_{((y_1,\ldots,y_1),\ldots,(y_n,\ldots,y_n))}:D\rightarrow
J^{k_1-1}_*(\bbR^m,\bbR)\times\cdots\times J^{k_n-1}_*(\bbR^m,\bbR)\] is
an isomorphism and we get a neighborhood of the form
$U_1^{k_1}\times\cdots\times U_n^{k_n}$ over which $G_D$ is an
isomorphism. Denoting the induced neighborhood of $Y$ in $\SP_d$ by
$W$ one gets an interpolation map $A:W\times\mcR\rightarrow D$.

\begin{proposition}
The interpolation map $A$ is continuous and $A(Y,f)=0$ whenever $f\in \mfm_Y$. \qed
\end{proposition}

\begin{proof}[Proof of Theorem~\ref{int_continuous_interpolation}]
Let $F\pitchfork\mcP_d$. Then for each $Y\in\SP_d$ we can choose a subspace $D\subseteq F$ complementary to $\mfm_Y$ to obtain a continuous interpolation map with values in $D\subseteq F$ defined in a neighbourhood of $Y$. Thus we have proved Theorem~\ref{int_continuous_interpolation} locally (in $\SP_d$).

To globalize we glue the local interpolation maps using a partition of unity on $\SP_d$. This is possible since the interpolation maps form an affine space. In this way we obtain a continuous interpolation map and with a bit of care we may achieve that for a fixed $Y$ we have $A(Y,f)=0$ for $f\in\mfm_Y$.
\end{proof}

\end{document}